\documentclass[onecolumn,11pt]{IEEEtran}
\usepackage{graphicx}
\usepackage{mdwmath,amsmath,amsthm}
\usepackage{amsfonts}
\usepackage[ruled,algonl,lined]{algorithm2e}

\newcommand{\p}{{\rm{P}}}
\newcommand{\n}{{\rm{N}}}

\newtheorem{theorem}{Theorem}

\newtheorem{example}{Example}
\newtheorem{remark}{Remark}
\theoremstyle{definition}
\ifCLASSINFOpdf
\else
\fi
\begin{document}
%
\title{Control and Detection of Discrete Spectral Amplitudes in Nonlinear Fourier Spectrum}
%
%
%

\author{Vahid Aref\\
Nokia Bell Labs, Stuttgart, Germany

}

\maketitle

\begin{abstract}
Nonlinear Fourier division Multiplexing (NFDM) can be realized from modulating 
the discrete nonlinear spectrum of an $N$-solitary waveform.
To generate an $N$-solitary waveform from 
desired discrete spectrum (eigenvalue and discrete spectral amplitudes), 
we use the Darboux Transform. We explain how to the norming factors must be set in order to
have the desired discrete spectrum. To derive these norming factors, we study the evolution 
of nonlinear spectrum by adding a new eigenvalue and its spectral amplitude. We further simplify
the Darboux transform algorithm.

We propose a novel algorithm (to the best of our knowledge) to numerically compute the nonlinear Fourier Transform (NFT) of a given pulse. The NFT algorithm, called forward-backward method,
is based on splitting the signal into two parts and computing the nonlinear spectrum of each part 
from boundary ($\pm\infty$) inward. The nonlinear spectrum (discrete and continuous) derived from efficiently combining both parts has a promising numerical precision.
This method can use any of one-step discretization NFT methods, e.g.
Crank-Nicolson, as an NFT kernel for the forward or backward part. Using trapezoid rule of integral,
we use an NFT kernel (we called here Trapezoid discretization NFT) in forward-backward method which results discrete spectral amplitudes with a very good numerical precision.
 
These algorithms, forward-backward method and Darboux transform, are used in \cite{aref2015experimental,aref2016designsoliton} for design and detection of phase-modulated 2-soliton pulses, and more recently, in \cite{buelow2016Transmission} for design and detection of more complex pulses with 7 eigenvalues and modulation of spectral phase.
For those soliton pulses, the discrete spectral amplitudes (in particular, phase) of both eigenvalues 
are quite precisely estimated using the forward-backward method.     

\end{abstract}


%
\IEEEpeerreviewmaketitle


\section{Preliminaries: Nonlinear Fourier Transform}

Consider the Nonlinear Schr{\"o}diger (NLS) equation,
\begin{equation}
\label{eq:nls}
\frac{\partial}{\partial z} q(t,z) -\frac{\partial^2}{\partial t^2} q(t,z) = 2\vert q(t,z) \vert^2 q(t,z). 
\end{equation}
Nonlinear Fourier Transform (NFT) (or inverse scattering method) is an efficient way 
to solve this nonlinear differential equation or track the evolution of its solutions.
NFT is a bijective mapping of a solution from time domain into different domain, called nonlinear spectrum. The spectral components transform in $z$ linearly. The inverse nonlinear Fourier transform (INFT) maps the spectral components to the corresponding solution in time domain.
We use Zakharov-Shabat system \cite{shabat1972exact,ablowitz1981solitons,yousefi2013thesis}(also known as AKNS system after the authors of \cite{akns1974inverse})
to characterize nonlinear spectrum. The NLS equation \eqref{eq:nls} is 
the compatibility condition of the following two linear equations:
\begin{equation}
\label{eq:vt}
\frac{\partial}{\partial t}v(t,z;q,\lambda) =
\begin{pmatrix}
-j\lambda & q(t,z)\\
-q^*(t,z) & j\lambda
\end{pmatrix}
v(t,z;q,\lambda),
\end{equation} 
and,
\begin{equation}
\label{eq:vz}
\frac{\partial}{\partial z}v(t,z;q,\lambda) =
\begin{pmatrix}
2j\lambda^2-j\vert q(t,z)\vert^2 & -2\lambda q(t,z)-jq_t(t,z)\\
2\lambda q^*(t,z)-jq^*_t(t,z) & -2j\lambda^2 + j\vert q\vert^2
\end{pmatrix}
v(t,z;q,\lambda),
\end{equation}
where $\lambda$ is the spectral parameter, $v(t,z;q,\lambda)$ is a $2\times 1$ vector, $q_t$ denotes $\partial q/\partial t$ and $q^*$ denotes complex conjugate of $q$. By differentiating 
\eqref{eq:vt} and \eqref{eq:vz} with respect to $z$ and $t$, the right sides become equal if 
$q(t,z)$ satisfies the NLS equation \eqref{eq:nls}. In particular,
this implies that the nonlinear Schr{\"o}dinger equation possess a ``hidden linearity'' in the form of
\eqref{eq:vt} and \eqref{eq:vz}.
The temporal linear operator \eqref{eq:vt} and 
the spatial linear operator \eqref{eq:vz} are called the Lax pair of \eqref{eq:nls}~\cite{yang2010nonlinear}.

We assume that $\vert q(t,z)\vert$ vanishes sufficiently fast as $\vert t\vert\to\infty$, or $z\to\infty$. Then,
\begin{equation*}
\frac{\partial}{\partial t}v(t,z;q,\lambda) \underset{|t|\to\infty}{\to}
\begin{pmatrix}
-j\lambda & 0\\
0 & j\lambda
\end{pmatrix}
v(t,z;q,\lambda).
\end{equation*}
It implies that $v(t,z;q,\lambda)\underset{|t|\to\infty}{\to}\left(c_1 e^{-j\lambda t}, c_2e^{j\lambda t}\right)^T$ for some constants $c_1,c_2\in\mathbb{C}$. Similarly,
$v(t,z;q,\lambda)\underset{z\to\infty}{\to}\left(z_1 e^{+j\lambda^2 z}, z_2e^{-j\lambda^2 z}\right)^T$. Let us focus on the temporal Equation \eqref{eq:vt} and review some results mainly taken from \cite{yousefi2013thesis}. For presentation simplicity, we omit dependency on $z$.
We consider $v(t;\lambda)$ as the elements of the vector space $\mathcal{V}$ equipped
with the Wronskian product,
\begin{equation*}
\mathcal{W}(v,u)=\text{det}\begin{pmatrix}
v_1&u_1\\
v_2&u_2
\end{pmatrix}=v_1u_2-u_1v_2.
\end{equation*}
Define the adjoint of any vector $v(t)\in\mathcal{V}$ as 
\begin{equation*}
\bar{v}(t)=\begin{pmatrix}
v_2^*\\
-v_1^*
\end{pmatrix}.
\end{equation*}
The following properties are shown in \cite[Lemma 7]{yousefi2013thesis}: 
\begin{itemize}
\item[i)] $u(t;\mu)=\bar{v}(t;\lambda=\mu^*)$ is a solution of \eqref{eq:vt} for $\lambda=\mu$ if 
$v(t;\mu^*)$ is its solution for $\lambda=\mu^*$.
\item[ii)] Let $v(t;\lambda)$ be a solution of \eqref{eq:vt}. $\mathcal{W}(\bar{v},v)=\vert v_1\vert^2+ \vert v_2\vert^2=c$ for some constant $c\in\mathbb{R}$ independent of $t$. And,
\item[iii)] if $\mathcal{W}(\bar{v},v)\ne 0$, then $\bar{v}(t;\lambda^*)$ and $v(t;\lambda)$ are linearly independent and hence, they form a basis for the solution subspace of $\lambda$.
\end{itemize}

We need some initial conditions to determine a specific temporal solution. 
Assume that $\lambda\in\mathbb{C}^+$, the upper complex plane.
Let define $\phi^\p(t;\lambda)$, $\bar{\phi}^\p(t;\lambda^*)$, $\phi^\n(t;\lambda)$
and $\bar{\phi}^\n(t;\lambda^*)$ with the following boundary conditions:
\begin{center}
\begin{tabular}{p{.5\linewidth} c}
$
\phi^\p(t;\lambda)\underset{t\to\infty}{\to}\begin{pmatrix}
0\\1
\end{pmatrix}e^{j\lambda t}$,
&
$
\phi^\n(t;\lambda)\underset{t\to-\infty}{\to}\begin{pmatrix}
1\\0
\end{pmatrix}e^{-j\lambda t}
$
\\
$
\bar{\phi}^\p(t;\lambda^*)\underset{t\to\infty}{\to}\begin{pmatrix}
1\\0
\end{pmatrix}e^{-j\lambda t}$,
&
$
\bar{\phi}^\n(t;\lambda^*)\underset{t\to-\infty}{\to}\begin{pmatrix}
0\\-1
\end{pmatrix}e^{j\lambda t}
$
\end{tabular}
\end{center}
We can see that $\mathcal{W}\left(\phi^\p(t;\lambda),\bar{\phi}^\p(t;\lambda^*)\right)=-1$ and
$\mathcal{W}\left(\phi^\n(t;\lambda),\bar{\phi}^\n(t;\lambda^*)\right)=-1$. Therefore both pairs of $(\phi^\p,\bar{\phi}^\p)$ and  $(\phi^\n,\bar{\phi}^\n)$ can be used as a basis for the 
solutions of \eqref{eq:vt} for $\lambda$ and are called \textit{canonical solutions}. Thus,
\begin{equation}
\label{eq:scatter}
\begin{pmatrix}
\phi^\n(t;\lambda),
\bar{\phi}^\n(t;\lambda^*) 
\end{pmatrix}
=
\begin{pmatrix}
\bar{\phi}^\p(t;\lambda^*),
\phi^\p(t;\lambda)
\end{pmatrix}
\begin{pmatrix}
a(\lambda) & b^*(\lambda^*)\\
b(\lambda)&-a^*(\lambda^*)
\end{pmatrix}
,
\end{equation}
where $a(\lambda)=\mathcal{W}(\phi^\n,\phi^\p)$ and $b(\lambda)=\mathcal{W}(\bar{\phi}^\p,\phi^\n)$. Note that $a(\lambda)$ and $b(\lambda)$ are independent of $t$. The matrix
$$
S=
\begin{pmatrix}
a(\lambda) & b^*(\lambda^*)\\
b(\lambda)&-a^*(\lambda^*)
\end{pmatrix}
$$
is called the scattering matrix and it captures the required information to identify $q(t)$. More
precisely, at $t\to-\infty$, where $q(t)$ is absent, $\phi^\n(t;\lambda)\to e^{-j\lambda t}(1,0)^T$.
For any finite $t_0$, $\phi^\n(t_0;\lambda)$ captures the information of $q(t)$, $t\in(-\infty,t_0)$ by integrating over $q(t)$ according to the temporal evolution \eqref{eq:vt}. As we will see,
the signal $q(t)$ can be uniquely described by knowing $a(\lambda)$ and $b(\lambda)$ for some particular $\lambda\in\mathbb{C}$~\cite{ablowitz1981solitons}.

If $q(t)$ decays faster than any polynomial, it is shown in \cite{ablowitz1981solitons} 
that $\phi^\n(t;\lambda)e^{j\lambda t}$ and $\phi^\p(t;\lambda)e^{-j\lambda t}$ are analytic in 
upper-half complex plane and consequently, $a(\lambda)=\mathcal{W}(\phi^\n,\phi^\p)$ is analytic in the same domain.
It implies that the zeros of $a(\lambda)$ are isolated. Note that 
$\bar{\phi}^\p e^{+j\lambda t}$ is only analytic in lower-half complex plane and therefore, $b(\lambda)=\mathcal{W}(\phi^\n,\bar{\phi}^\p)$ is not analytic in general.

Moreover, the projection \eqref{eq:scatter} is well-defined for $\lambda\in\mathbb{R}$ as the canonical solutions are bounded. We already explained that we can make solutions for $\lambda^*$ from the solutions for $\lambda$. Therefore, we can focus on upper or lower half of complex plane. For $\lambda\in\mathbb{C}^+$, $\phi^\n$ and $\bar{\phi}^\p$ are unbounded as $t\to\infty$. However,
the projection \eqref{eq:scatter} is still well-defined when $a(\lambda)=0$. These $\lambda$ 
(and $\lambda^*$)  are the eigenvalues of the corresponding spectral operator:
\begin{equation}
\label{eq:L}
j\begin{pmatrix}
\frac{\partial}{\partial t} & -q(t,z)\\
-q^*(t,z)& -\frac{\partial}{\partial t} 
\end{pmatrix}
v=\lambda v.
\end{equation}
This can be immediately derived from reordering the temporal equation \eqref{eq:vt}.
The Zakharov-Shabat system has two types of spectrum: the discrete spectrum and the 
continuous spectrum. The latter is defined as:
\begin{equation*}
Q_c(\lambda) = \frac{b(\lambda)}{a(\lambda)}, 
\end{equation*}
for $\lambda\in\mathbb{R}$. The continuous spectrum corresponds to the \textit{dispersive} part of 
signal $q(t)$. The descrete spectrum is defined over zeros of $a(\lambda)$:
\begin{equation*}
Q_d(\lambda_i) = \frac{b(\lambda_i)}{a'(\lambda_i)}, 
\end{equation*}
where $a'(\lambda_i)=\frac{d}{d\lambda}a(\lambda)\vert_{\lambda=\lambda_i}$
for all $\{\lambda_1,\dots,\lambda_n\}\subset\mathbb{C}^+\setminus\mathbb{R}$
the finite set satisfying $a(\lambda)=0$.
The discrete spectrum corresponds to the solitonic part of signal $q(t)$. Here, the spectral coefficients $a(\lambda)$ and $b(\lambda)$ are given by
\begin{align}
a(\lambda)&=\lim_{t\to+\infty}e^{+j\lambda t}\phi^\n_1,\nonumber \\
b(\lambda)&=\lim_{t\to+\infty}e^{-j\lambda t}\phi^\n_2
\label{eq:ab}
\end{align}
The signal $q(t)$ can be characterized from $Q_c(\lambda),\lambda\in\mathbb{R}$ and the set of $(\lambda_i,Q_d(\lambda_i))$ using an inverse nonlinear Fourier transform such as Riemann-Hilbert method~\cite{ablowitz1981solitons,yousefi2013thesis,yang2010nonlinear}. In the next section, we introduce another method based on Darboux Transform. This method is very useful when $q(t)$ has only discrete spectrum.

\section{Inverse Transform From Darboux Transform}

The Darboux transform (DT) was first introduced in the context of Sturm-Liouville differential equations. Later, it was extended in the nonlinear integrable
systems, provides the possibility to construct from one solution of an integrable
equation another solution~\cite{matveev1991darboux}. 
It has been shown in \cite{yousefi2013thesis} how we can recursively construct an $N-$solitary waveform
using the Darboux transform. The computational complexity of this method is $O(N^2)$ which is
much smaller than $O(N^3)$ complexity of Riemann-Hilbert method. The complexity can be reduced to almost linear complexity using the techniques in \cite{Wahls2016fastINFT}. 
Consider an $N-$solitary waveform with discrete spectrum $(\lambda_i,Q_d(\lambda_i))$ for $i=1,...,N$. 
The algorithm constructs the signal from initial parameters $(\lambda_i)$ and 
$(A_i e^{-j\lambda_i t},B_i e^{+j\lambda_i t})^T$ for 
$i=1,...,N$. It starts from $q^{(0)}=0$ and recursively adds a new eigenvalue to the 
current waveform.
Although the final waveform has the $N$ desired eigenvalues, 
it was not clear in \cite{yousefi2013thesis} how $(A_i,B_i)$ must be chosen to have the desired $Q_d(\lambda_i)$.
In this section, we derive a closed form solution for $(A_i,B_i)$ which is given in the 
following Algorithm~\ref{alg:DT1}.
\begin{algorithm}
\SetKwInOut{Input}{Input}
\SetKwInOut{Output}{Output}
\Input{$N$ Discrete spectral values $(\lambda_i,Q_d(\lambda_i))$ for $i=1,\dots,N$.}
\Output{Construct $N-$solitary waveform $q(t)$ from given discrete spectrum.}
\BlankLine
\Begin{
\tcc{initialization}
\For{$i\leftarrow 1$ \KwTo $N$}{
$A_i\longleftarrow 1$\;
$B_i\longleftarrow-\frac{Q_d(\lambda_i)}{\lambda_i-\lambda_i^*}\prod_{k=1,k\ne i}^N \frac{\lambda_i-\lambda_k}{\lambda_i-\lambda_k^*}$\;
$v^{(0)}_i(t)\longleftarrow(A_i e^{-j\lambda_i t},B_i e^{j\lambda_i t})^T$\;
}
$q^{(0)} \longleftarrow 0$\;
\tcc{recursively add $(\lambda_i,Q_d(\lambda_i))$ to $q^{(i-1)}(t)$}
\For{$i\leftarrow 1$ \KwTo $N$}{
$(\psi_1,\psi_2)\longleftarrow v_i^{(i-1)}(t)$\;
$q^{(k)}(t)\longleftarrow q^{(i-1)}(t)-2j(\lambda_i-\lambda_i^*)\frac{\psi_2^*(t)\psi_1(t)}{|\psi_1(t)|^2+|\psi_2(t)|^2}$\tcc*[l]{signal update}
\For{$k\leftarrow i+1$ \KwTo $N$}{
$v_{k,1}^{(i)}(t)\longleftarrow \left(\lambda_k -\lambda_i^*-\frac{(\lambda_i-\lambda_i^*)|\psi_1(t)|^2}{|\psi_1(t)|^2+|\psi_2(t)|^2}
\right)v_{k,1}^{(i-1)}(t) -\frac{(\lambda_i-\lambda_i^*)\psi_2^*(t)\psi_1(t)}{|\psi_1(t)|^2+|\psi_2(t)|^2}v_{k,2}^{(i-1)}(t)$\;
$v_{k,2}^{(i)}(t)\longleftarrow 
-\frac{(\lambda_i-\lambda_i^*)\psi_2(t)\psi_1^*(t)}{|\psi_1(t)|^2+|\psi_2(t)|^2}v_{k,1}^{(i-1)}(t)
+\left(\lambda_k -\lambda_i+\frac{(\lambda_i-\lambda_i^*)|\psi_1(t)|^2}{|\psi_1(t)|^2+|\psi_2(t)|^2}
\right)v_{k,2}^{(i-1)}(t)$\;
$v_{k}^{(i)}(t)\longleftarrow\left(v_{k,1}^{(i)}(t),v_{k,2}^{(i)}(t)\right)$\;

}
}
}
\caption{Inverse Nonlinear Fourier Transform from Darboux Transform\label{alg:DT1}}
\end{algorithm}

This algorithm can be simplified further. Algorithm~\ref{alg:DT2} construct the waveform from updating the ratios $\rho^{(k)}_i=-v_{i,2}^{(k)}(t)/v_{i,1}^{(k)}(t)$. Then both the computational and space complexity become almost half. 

\begin{algorithm}
\SetKwInOut{Input}{Input}
\SetKwInOut{Output}{Output}
\Input{$N$ Discrete spectral values $(\lambda_i,Q_d(\lambda_i))$ for $i=1,\dots,N$.}
\Output{Construct $N-$solitary waveform $q(t)$ from given discrete spectrum.}
\BlankLine
\Begin{
\tcc{initialization}
\For{$i\leftarrow 1$ \KwTo $N$}{
$\rho_i^{(0)}(t)\longleftarrow\left(\frac{Q_d(\lambda_i)}{\lambda_i-\lambda_i^*}\prod_{k=1,k\ne i}^N \frac{\lambda_i-\lambda_k}{\lambda_i-\lambda_k^*}\right) e^{2j\lambda_i t}$\;
}
$q^{(0)} \longleftarrow 0$\;
\tcc{recursively add $(\lambda_k,Q_d(\lambda_k))$ to $q^{(k-1)}(t)$}
\For{$i\leftarrow 1$ \KwTo $N$}{
$\rho(t)\longleftarrow \rho_i^{(i-1)}(t)$\;
$q^{(i)}(t)\longleftarrow q^{(i-1)}(t)+2j(\lambda_i-\lambda_i^*)\frac{\rho^*(t)}{1+|\rho(t)|^2}$\tcc*[l]{signal update}
\For{$k\leftarrow i+1$ \KwTo $N$}{
$\rho_k^{(i)}(t) \longleftarrow  
\frac{(\lambda_k -\lambda_i)\rho_{k}^{(i-1)}(t)  +\frac{\lambda_i-\lambda_i^*}{1+|\rho(t)|^2}
(\rho_{k}^{(i-1)}(t)-\rho(t))}
{
\lambda_k -\lambda_i^*-\frac{\lambda_i-\lambda_i^*}{1+|\rho(t)|^2}\left(1 + \rho^*(t)\rho_{k}^{(i-1)}(t) \right)
}$\;
}
}
}
\caption{Inverse Nonlinear Fourier Transform from Darboux Transform\label{alg:DT2}}
\end{algorithm}

Now, we introduce the underlying Darboux Transform. The following theorem
shows how the spectrum changes by applying the Darboux transform. 
The first part explains how the signal must be modified to include a new 
eigenvalue and how the canonical eigenvectors are consequently modified.
This part is the same as Theorem 11 in \cite{yousefi2013thesis} and the proof is
repeated for completeness. As a result, 
the second part shows how the spectrum is modified. It leads to the initial 
values of $(A_i,B_i)$. The proof is rather similar to \cite[Theorem 1]{lin1990evolution}.
\begin{theorem} Consider an absolutely integrable $q(t,z)$ with the scattering
functions $a(\lambda)$ and $b(\lambda)$ at $z=0$.
Assume that the nontrivial $(\psi_1(t,z),\psi_2(t,z))^T$ fulfills \eqref{eq:vt} and
\eqref{eq:vz} for some $\lambda_0$ such that $(\lambda_0-\lambda_0^*)>0$ and $a(\lambda_0)\neq 0$ (it is not an eigenvalue), i.e.
\begin{align}
\left(\frac{\partial}{\partial t}\psi_1(t,z),\frac{\partial}{\partial t}\psi_2(t,z)\right)^T&=P\left(q(t,z),\lambda_0\right)\left(\psi_1(t,z),\psi_2(t,z)\right)^T\nonumber\\
\left(\frac{\partial}{\partial z}\psi_1(t,z),\frac{\partial}{\partial z}\psi_2(t,z)\right)^T&=M\left(q(t,z),\lambda_0\right)\left(\psi_1(t,z),\psi_2(t,z)\right)^T,
\label{eq:psi_updates}
\end{align}
where $P\left(q(t,z),\lambda_0\right)$ and $M\left(q(t,z),\lambda_0\right)$
are given in \eqref{eq:vt} and \eqref{eq:vz}, respectively. Let 
$\Sigma=\Psi \Lambda_0 \Psi^{-1}$, where 
$$
\Psi=
\begin{pmatrix}
\psi_1&\psi_2^*\\
\psi_2&-\psi_1^*
\end{pmatrix},\hspace*{1cm}
\Lambda_0=
\begin{pmatrix}
\lambda_0&0\\
0&\lambda_0^*
\end{pmatrix}.
$$
Define $\tilde{q}(t,z)=q(t,z)-2j(\lambda_0-\lambda_0^*)\frac{\psi_2^*(t,z)\psi_1(t,z)}{|\psi_1(t,z)|^2+ |\psi_2(t,z)|^2}$. Thus,
\begin{itemize}
\item[(a)] If $v(t,z;q,\mu)$ satisfies \eqref{eq:vt} and \eqref{eq:vz}, then
$u(t,z;\tilde{q},\mu)=\left(\mu I-\Sigma\right)v(t,z;q,\mu)$ satisfies 
\eqref{eq:vt} and \eqref{eq:vz} for $\tilde{q}(t,z)$ and spectral parameter $\mu$.
\item[(b)] Assume that there are non-zero constants
\begin{align}
\label{eq:psi_limits}
\lim_{t\to+\infty}e^{j\lambda_0 t}\psi_1(t)=\psi_1^{+},\lim_{t\to-\infty}e^{-j\lambda_0 t}\psi_2(t)=\psi_2^{-}
\end{align}
(They are not zero as  $\lambda_0$ is not an eigenvalue of $q$).
The discrete spectrum of $\tilde{q}(t,z)$ consists of the eigenvalues of 
$q(t,z)$ and also $\lambda_0$ (and $\lambda_0^*$). Let $\tilde{a}(\lambda)$ and $\tilde{b}(\lambda)$ denote the scattering functions of $\tilde{q}(t)\doteq \tilde{q}(t,z=0)$. Then,
\begin{align*}
\tilde{a}(\lambda)&=\frac{\lambda-\lambda_0}{\lambda-\lambda^*_0}
a(\lambda),\text{ for }\lambda\in\mathbb{C}^+,\\
\tilde{b}(\lambda)&=b(\lambda) \text{ for }\lambda\in\mathbb{R}\cup\{\text{eigenvalues of }q(t,z)\},\\
b(\lambda_0)&=-\frac{\psi_2^-}{\psi_1^+},\\
\tilde{Q}_c(\lambda)&=\frac{\lambda-\lambda^*_0}{\lambda-\lambda_0}Q_c(\lambda),\\
\tilde{Q}_d(\lambda_i)&=\frac{\lambda_i-\lambda^*_0}{\lambda_i-\lambda_0}Q_d(\lambda_i)\text{ for }\lambda_i\in\{\text{eigenvalues of }q(t,z)\},\\
\tilde{Q}_d(\lambda_0)&=
 -\frac{\lambda_0-\lambda_0^*}{a(\lambda_0)}\times\frac{\psi_2^-}{\psi_1^+}
\end{align*}
where $\tilde{Q}_c(\cdot)$ and $\tilde{Q}_d(\cdot)$ denote the
 continuous and discrete spectrum of $\tilde{q}$.
\end{itemize}
\end{theorem}
\begin{proof}
Let us begin from part $(a)$. We already discussed that if $(\psi_1,\psi_2)^T$ satisfies \eqref{eq:vt} and \eqref{eq:vz} for some $\lambda$, then  its adjoint $(\psi^*_2,-\psi^*_1)$
satisfies \eqref{eq:vt} and \eqref{eq:vz} for $\lambda^*$. Here, we show $u(t,z;\tilde{q},\mu)$
fulfills the temporal equation \eqref{eq:vt} and fulfilling the spatial equation similarly follows.
Define
\begin{equation*}
J=j
\begin{pmatrix}
-1&0\\
0&1
\end{pmatrix}, \text{ and }
Q=\begin{pmatrix}
0&q(t,z)\\
-q^*(t,z) &0
\end{pmatrix}.
\end{equation*}
Then from \eqref{eq:vt}, $v_t = \mu J v+Qv$ and $\Psi_t=J\Psi\Lambda_0 + Q\Psi$. Thus,
\begin{align*}
\Sigma_t & = \Psi_t\Lambda_0\Psi^{-1} + \Psi\Lambda_0(\Psi^{-1})_t= \Psi_t\Lambda_0\Psi^{-1} - \Psi\Lambda_0\Psi^{-1}\Psi_t\Psi^{-1}\\
& = J\Psi\Lambda_0^2\Psi^{-1} + Q\Psi\Lambda_0\Psi^{-1} - \Psi\Lambda_0\Psi^{-1} (J\Psi\Lambda_0 + Q\Psi)\Psi^{-1}\\
& = J\Sigma^2 +Q\Sigma - \Sigma J \Sigma - \Sigma Q.
\end{align*}
From definition of $u$, we have
\begin{align*}
u_t &= (\mu I -\Sigma)v_t -\Sigma_t v = (\mu I -\Sigma)(\mu J v + Qv)-J\Sigma^2v-Q\Sigma v
+ \Sigma J \Sigma v + \Sigma Qv\\
&=\mu J (\mu I)v +\mu Qv-\mu \Sigma Jv -\Sigma Q v -J\Sigma^2v-Q\Sigma v
+ \Sigma J \Sigma v + \Sigma Qv\\
&= \mu J (\mu I-\Sigma + \Sigma)v + \mu Qv - Q\Sigma v -\mu\Sigma J v -(J\Sigma-\Sigma J)\Sigma v\\
&=\mu J (\mu I-\Sigma)v + Q(\mu I-\Sigma)v + \mu J\Sigma v -\mu\Sigma J v -(J\Sigma-\Sigma J)\Sigma v\\
&=\mu J (\mu I-\Sigma)v + Q(\mu I-\Sigma)v + (J\Sigma-\Sigma J)(\mu I-\Sigma)v\\
&= (\mu J + Q + J\Sigma-\Sigma J )u.
\end{align*}
We have $u_t=(\mu J + \tilde{Q})u$, where 
\begin{equation*}
\tilde{Q}=Q+J\Sigma-\Sigma J = \begin{pmatrix}
0&q(t,z)\\
-q^*(t,z)&0
\end{pmatrix}
+
\begin{pmatrix}
0&-2j(\lambda_0-\lambda_0^*)\psi_1\psi_2^*\\
2j(\lambda_0-\lambda_0^*)\psi_1^*\psi_2&0.
\end{pmatrix}
\end{equation*}
Now we prove part $(b)$. It is enough to compute the change of scattering functions 
at some $z$, let say $z=0$. Then the evolution immediately follows from $a(\lambda,z)=a(\lambda,z=0)$ and $b(\lambda,z)=b(\lambda,z=0)\exp(-4j\lambda^2z)$. 
We have
\begin{equation*}
\Sigma = \frac{1}{|\psi_1|^2+ |\psi_2|^2}
\begin{pmatrix}
\lambda_0|\psi_1|^2 + \lambda^*_0|\psi_2|^2
& (\lambda_0-\lambda^*_0)\psi_2^*\psi_1\\
(\lambda_0-\lambda^*_0)\psi_1^*\psi_2
& \lambda_0|\psi_2|^2 + \lambda^*_0|\psi_1|^2
\end{pmatrix}
\end{equation*}

Note that $(\psi_1(t),\psi_2(t))^T$ satisfies \eqref{eq:vt} for $\lambda_0$ which implies that,
\begin{align*}
\begin{pmatrix}
\psi_1(t)\\
\psi_2(t)
\end{pmatrix}
\underset{t\to{\pm}\infty}{\sim}
\begin{pmatrix}
c_1^{\pm} e^{-j\lambda_0 t}\\ 
c_2^{\pm}e^{+j\lambda_0 t}
\end{pmatrix},
\end{align*}
for some constants $c_1^{\pm}$ and $c_2^{\pm}$. Therefore, 
\begin{equation}
\label{eq:sigma_limit}
\lim_{t\to-\infty}\Sigma=
\begin{pmatrix}
\lambda_0^*&0\\
0&\lambda_0
\end{pmatrix}, 
\lim_{t\to+\infty}\Sigma=
\begin{pmatrix}
\lambda_0&0\\
0&\lambda_0^*
\end{pmatrix}
\end{equation}

We study how the canonical eigenvector modifies. Consider first spectral parameter $\lambda\neq\lambda_0$. One can verify from \eqref{eq:sigma_limit} that
\begin{align}
\phi^\n(t;\lambda,\tilde{q})&=
\frac{1}{\lambda-\lambda^*_0}(\lambda I- \Sigma)\phi^\n(t;\lambda,q),\nonumber\\
\phi^\p(t;\lambda,\tilde{q})&=
\frac{1}{\lambda-\lambda^*_0}(\lambda I- \Sigma)\phi^\p(t;\lambda,q).
\label{eq:EV_evolution}
\end{align}
From the definition, $\lim_{t\to\infty} \phi_1^\n(t;\lambda,q)\exp(j\lambda t)=a(\lambda)$ and
 $\lim_{t\to\infty} \phi_2^\n(t;\lambda,q)\exp(-j\lambda t)=b(\lambda)$. Thus, 
\begin{equation*}
\phi^\n_1(t;\lambda,\tilde{q})e^{j\lambda t}=
\frac{1}{\lambda-\lambda_0^*}\left(
\frac{(\lambda-\lambda_0)|\psi_1|^2 + (\lambda-\lambda_0^*)|\psi_2|^2}{|\psi_1|^2+ |\psi_2|^2}
\phi^\n_1(t;\lambda,q)e^{j\lambda t} 
-
\frac{(\lambda-\lambda_0^*)\psi_2^*\psi_1}{|\psi_1|^2+ |\psi_2|^2}
\phi^\n_2(t;\lambda,q)e^{j\lambda t}\right)
\end{equation*}
Since each term in the right side has a limit as $t\to\infty$, we have
\begin{align*}
\tilde{a}(\lambda)&= \lim_{t\to\infty} \phi_1^\n(t;\lambda,\tilde{q})\exp(j\lambda t)
=\frac{1}{\lambda-\lambda_0^*} 
\lim_{t\to\infty}\frac{(\lambda-\lambda_0)|\psi_1|^2 + (\lambda-\lambda_0^*)|\psi_2|^2}{|\psi_1|^2+ |\psi_2|^2}
\lim_{t\to\infty}\phi^\n_1(t;\lambda,q)e^{j\lambda t}\\ 
&= \frac{\lambda-\lambda_0}{\lambda-\lambda_0^*}
a(\lambda)
\end{align*}
Since $\tilde{a}(\lambda)$ is analytic 
for $\lambda\in\mathbb{C}^+$, 
and therefore, $\tilde{a}(\lambda_0)=0$.
It implies that $\tilde{q}(t)$ consists of $\lambda_0$ and the eigenvalues of $q(t)$. 
Now we compute $b(\lambda)$ for $\lambda\in\mathbb{R}$
and for the eigenvalues. Let first assume that 
$\lambda\in\mathbb{R}$. Then,
\begin{align*}
\phi^\n_2(t;\lambda,\tilde{q})e^{-j\lambda t}=
\frac{1}{\lambda-\lambda_0^*}\left(
-
\frac{(\lambda-\lambda_0^*)\psi_2^*\psi_1}{|\psi_1|^2+ |\psi_2|^2}
\phi^\n_1(t;\lambda,q)e^{-j\lambda t}+
\frac{(\lambda-\lambda_0^*)|\psi_1|^2 + (\lambda-\lambda_0)|\psi_2|^2}{|\psi_1|^2+ |\psi_2|^2}
\phi^\n_2(t;\lambda,q)e^{-j\lambda t} 
\right)
\end{align*}
Since $|e^{-j\lambda t}|=1$,
we can see that $\phi^\n_1(t;\lambda,q)e^{-j\lambda t}=
e^{-2j\lambda t}(\phi^\n_1(t;\lambda,q)e^{+j\lambda t})$ is bounded as $t\to\infty$ and thus,
\begin{align*}
\tilde{b}(\lambda)&= \lim_{t\to\infty} \phi_2^\n(t;\lambda,\tilde{q})\exp(-j\lambda t)
=\frac{1}{\lambda-\lambda_0^*} 
\lim_{t\to\infty}\frac{(\lambda-\lambda_0^*)|\psi_1|^2 + (\lambda-\lambda_0)|\psi_2|^2}{|\psi_1|^2+ |\psi_2|^2}
\lim_{t\to\infty}\phi^\n_2(t;\lambda,q)e^{-j\lambda t}\\ 
&= b(\lambda)
\end{align*}
Now let $\lambda=\lambda_j$ an eigenvalue of $q(t)$. Then, $\tilde{a}(\lambda_j)=a(\lambda_j)=0$ and we know that $\phi^\n(t;\lambda_j,q)=b(\lambda_j)\phi^\p(t;\lambda_j,q)$ and 
$\phi^\n(t;\lambda_j,\tilde{q})=\tilde{b}(\lambda_j)\phi^\p(t;\lambda_j,\tilde{q})$. From \eqref{eq:EV_evolution}, $\tilde{b}(\lambda_j)=b(\lambda_j)$. For $\lambda=\lambda_0$,
\begin{align*}
u(t;\lambda_0,\tilde{q})&=\frac{\lambda_0-\lambda_0^*}{|\psi_1|^2+|\psi_2|^2}
\begin{pmatrix}
|\psi_2|^2&-\psi_1\psi_2^*\\
-\psi_1^*\psi_2&|\psi_1|^2
\end{pmatrix}
v(t;\lambda_0,q)\\
&=\frac{\lambda_0-\lambda_0^*}{|\psi_1|^2+|\psi_2|^2} 
\begin{pmatrix}
-\psi_2^*\\
\psi_1^*
\end{pmatrix}\mathcal{W}\left((\psi_1,\psi_2)^T,v\right),
\end{align*}
where $\mathcal{W}(\cdot)$ is the Wronskian product. If $v(t;\lambda_0,q)$ fulfills \eqref{eq:vt},
then we can easily verify that $\mathcal{W}\left((\psi_1,\psi_2)^T,v\right)=C$ a non-zero constant independent of $t$ if  $v(t;\lambda_0,q)\neq \alpha(\psi_1,\psi_2)^T$, $\alpha\in\mathbb{C}$~\cite[Lemma~7]{yousefi2013thesis}. Thus,
\begin{align*}
u(t;\lambda_0,\tilde{q})=\frac{C(\lambda_0-\lambda_0^*)}{|\psi_1|^2+|\psi_2|^2} 
\begin{pmatrix}
-\psi_2^*\\
\psi_1^*
\end{pmatrix}.
\end{align*}
Taking the limit and we can define from \eqref{eq:psi_limits},
\begin{equation*}
\phi^\n(t;\lambda_0,\tilde{q})=\frac{\psi_2^-}{|\psi_1|^2+|\psi_2|^2}
\begin{pmatrix}
\psi_2^*\\
-\psi_1^*
\end{pmatrix}.
\end{equation*}
One can directly verify the above eigenvector in \eqref{eq:vt}. Therefore,
\begin{align*}
b(\lambda_0)& =\lim_{t\to+\infty}\phi_2^\n(t;\lambda_0,\tilde{q})e^{-j\lambda_0 t}
=\lim_{t\to+\infty} \frac{-\psi_2^-}{|\psi_1|^2+|\psi_2|^2}\psi_1^*e^{-j\lambda_0 t}\\
&=-\frac{\psi_2^-}{\psi_1^+}.
\end{align*}
Now we can compute continuous and discrete spectrum from $\tilde{a}(\lambda)$ and 
$\tilde{b}(\lambda)$. For $\lambda\in\mathbb{R}$,
\begin{align*}
\tilde{Q}_c(\lambda) &= \frac{\tilde{b}(\lambda)}{\tilde{a}(\lambda)} =
\frac{\lambda-\lambda_0^*}{\lambda-\lambda_0}\frac{b(\lambda)}{a(\lambda)}\\
&=\frac{\lambda-\lambda_0^*}{\lambda-\lambda_0} Q_c(\lambda)
\end{align*}
For $\lambda$ chosen from the eigenvalues of $q$, $\frac{d}{d\lambda}\tilde{a}(\lambda)|_{\lambda=\lambda_j} = \frac{\lambda_j-\lambda_0}{\lambda_j-\lambda_0^*}\times\frac{d}{d\lambda}a(\lambda)|_{\lambda=\lambda_j}$ and hence,
\begin{equation*}
\tilde{Q}_d(\lambda_j) = \frac{\lambda_j-\lambda_0^*}{\lambda_j-\lambda_0} Q_d(\lambda_j).
\end{equation*}
At $\lambda=\lambda_0$, $\frac{d}{d\lambda}\tilde{a}(\lambda)|_{\lambda=\lambda_0} = \frac{a(\lambda_0)}{\lambda_0-\lambda_0^*}$ and
\begin{equation*}
\tilde{Q}_d(\lambda_0) = -\frac{\lambda_0-\lambda_0^*}{a(\lambda_0)}\times\frac{\psi_2^-}{\psi_1^+} .
\end{equation*}
\end{proof}
From Theorem 1, we can easily show the initial conditions $(A_i,B_i)$ in Algorithm~\ref{alg:DT1}. Let $(a^{(i)}(\lambda),b^{(i)}(\lambda))$ denote the scattering functions
of $q^{(i)}(t)$. For $q^{(0)}=0$, $a^{(0)}(\lambda)=1$. 
Consider the recursion $i,i\ge 1$.
From Algorithm~\ref{alg:DT1},
\begin{equation*}
\begin{pmatrix}
\psi_1\\
\psi_2
\end{pmatrix}=v_i^{(i-1)}(t) = (\lambda_i I-\Sigma_{i-1})(\lambda_i I-\Sigma_{i-2})\cdots(\lambda_i I-\Sigma_{1})\times
\begin{pmatrix}
A_i e^{-j\lambda_i t}\\
B_i e^{+j\lambda_i t}
\end{pmatrix}.
\end{equation*}
Although $\Sigma_{k}$ depends on $v_k^{(k-1)}(t)$, its limits are only a function of $\lambda_k$
(see \eqref{eq:sigma_limit}), i.e. 
\begin{equation*}
\lim_{t\to-\infty}\Sigma_k=
\begin{pmatrix}
\lambda_k^*&0\\
0&\lambda_k
\end{pmatrix}, 
\lim_{t\to+\infty}\Sigma_k=
\begin{pmatrix}
\lambda_k&0\\
0&\lambda_k^*
\end{pmatrix}
\end{equation*}
It implies that $\psi_1^+=\prod_{k=1}^{i-1}(\lambda_i-\lambda_k)A_i$, and $\psi_2^-=\prod_{k=1}^{i-1}(\lambda_i-\lambda_k)B_i$. They are finite non-zero values, and 
we can apply part $(b)$ of Theorem 1. We can verify that $a^{(i)}=\prod_{k=1}^i\frac{\lambda-\lambda_k}{\lambda-\lambda_k^*}$ and $b^{(i)}(\lambda_i)=-\frac{\psi_2^-}{\psi_1^+}=-\frac{B_i}{A_i}$.
At the end, $b^{(N)}(\lambda_i)=b^{(i)}(\lambda_i)$ ($b(\lambda)$ at eigenvalues does not change
by applying the Darboux transform) and $a^{(N)}=\prod_{k=1}^N\frac{\lambda-\lambda_k}{\lambda-\lambda_k^*}$. It implies that
\begin{equation*}
Q_d(\lambda_i)= (\lambda_i-\lambda_i^*)\prod_{k=1,k\ne i}^N \frac{\lambda_i-\lambda_k^*}{\lambda_i-\lambda_k}\times\frac{-B_i}{A_i}.
\end{equation*}
Letting $A_i=1$, we have the desired $B_i$.

\section{Nonlinear Fourier Transform from Forward-Backward Method}

The forward-backward method is a general method which can be applied to 
any integration based algorithms introduced in \cite{yousefi2013thesis}  
such as forward and central discretizations, Crank-Nicolson method and Ablowitz-Ladik
algorithm. Before describing forward-backward method, 
we propose a new integration based algorithm 
which has a simple representation and 
estimates the spectrum with higher precision than the above mentioned algorithms.
We show that the forward discretization and Crank-Nicolson methods are the first order and the second order Taylor expansion of the new algorithm.
The complexity of this algorithm is quadratic in the number of samples which can be reduced applying techniques in \cite{Wahls2016fastNFT}. 

\subsection{Numerical Trapezoid Discretization}

Consider the temporal differential equation \eqref{eq:vt} and  its canonical 
solution $\phi^\n(t;\lambda)$.
\begin{equation}
\label{eq:phit}
\frac{\partial}{\partial t}\phi^\n(t;\lambda) =
\begin{pmatrix}
-j\lambda & q(t)\\
-q^*(t) & j\lambda
\end{pmatrix}
\phi^\n(t;\lambda).
\end{equation}  
The Nonlinear spectrum of signal $q(t)$ is characterized by the continuous part  
\begin{equation*}
Q_c(\lambda)=\frac{b(\lambda)}{a(\lambda)},
\end{equation*}
for $\lambda\in\mathcal{R}$ and the discrete part
\begin{equation*}
Q_d(\lambda_i) = \frac{b(\lambda_i)}{a'(\lambda_i)}, 
\end{equation*}
where $a'(\lambda_i)=\frac{d}{d\lambda}a(\lambda)\vert_{\lambda=\lambda_i}$
for all $\{\lambda_1,\dots,\lambda_n\}\subset\mathbb{C}^+\setminus\mathbb{R}$
the finite set satisfying $a(\lambda)=0$.
The discrete spectrum corresponds to the solitonic part of signal $q(t)$. The spectral coefficients $a(\lambda)$ and $b(\lambda)$ are given by
\begin{equation*}
a(\lambda)=\lim_{t\to+\infty}e^{+j\lambda t}\phi^\n_1,b(\lambda)=\lim_{t\to+\infty}e^{-j\lambda t}\phi^\n_2
\end{equation*}
In optical communication, each transmitted symbol has a finite support, thus
We assume that $q(t)$ is truncated such that $q(t)=0$ for $t\notin[-T_0,T_0]$.
Note that we assume that the support is symmetric around zero. 
This can be done by taking a larger symmetric support or more efficiently,  
by translation in time which modifies the nonlinear spectrum by a multiplicative factor.
Therefore $\phi^\n(t;\lambda)=(1,0)^Te^{-j\lambda t}$ for $t\leq -T_0$. 
Let 
$$
\psi(t;\lambda)=\begin{pmatrix}
\psi_1\\
\psi_2
\end{pmatrix}=
\begin{pmatrix}
\phi^\n_1e^{+j\lambda t}\\
\phi^\n_2e^{-j\lambda t}
\end{pmatrix}
$$
Then $\psi(t;\lambda)=(1,0)^T$ for $t\leq -T_0$ and the temporal equation \eqref{eq:phit} 
is changed to
\begin{equation}
\label{eq:psit}
\frac{\partial}{\partial t}\psi(t;\lambda) =
\begin{pmatrix}
0 & q(t)e^{+j2\lambda t}\\
-q^*(t)e^{-j2\lambda t} & 0
\end{pmatrix}
\psi(t;\lambda),
\end{equation}
and $a(\lambda)=\lim_{t\to\infty}\psi_1(t)=\psi_1(T_0)$ and $b(\lambda)=\lim_{t\to\infty}\psi_2(t)=\psi_2(T_0)$. 

We use the trapezoid rule to numerically compute $\psi(T_0;\lambda)$. 
Let us first explain this method in the following example.

\begin{example}
Consider the following scalar ordinary differential equation:
\begin{equation*}
\frac{d}{dt} x(t)=f(t)x(t), \text{ with initial condition } x(0)=1.
\end{equation*}
Its unique solution is 
$$
x(T)=\exp(\int_{0}^T f(\tau)\text{d}\tau)
$$
Now we use the trapezoid rule to estimate the integral. Let $t_n=hn$, $0\leq n\leq N$ where $h=T/N$, then
$$
\int_{0}^t f(\tau)\text{d}\tau = h\sum_{n=0}^{N}\frac{1}{2}(f(t_n)+f(t_{n+1})) + R_e
$$

and the error term is $R_e=\frac{T^3}{12N^2}f''(t_c)$, for some $t_c\in[0,T]$. Neglecting the error
term, we can estimate $x(T)\approx x_N$ recursively as follows:
\begin{equation}
\label{eq:scalerupdate}
x_{n+1} = e^{\frac{h}{2}f(t_{n+1})}e^{\frac{h}{2}f(t_{n})}x_n
\end{equation}

Defining $y_n=e^{\frac{h}{2}f(t_{n})}x_n$,
$$
y_{n+1}=e^{hf(t_{n+1})}y_n
$$
and $y_0=\exp({\frac{h}{2}f(0)})$. Note that $x_{n+1}$ is the exact solution of following differential equation at $t=t_n$,
\begin{equation*}
\frac{d}{dt}x(t)=\tilde{f}(t)x(t), \text{ with initial condition } x(0)=1,
\end{equation*}
and $\tilde{f}(t)$ is a step function defined as
\begin{equation*}
\tilde{f}(t)=
\begin{cases}
f(0)&0\leq t<h/2\\
f(t_n)& nh-h/2\leq t<nh+h/2 \text{ and } 1\leq n<N\\
f(T)& T-h/2\leq t \leq T
\end{cases}
\end{equation*}
The precision of this method is shown for two test functions $f(t)=t$ and $f(t)=sech(t)$ in Figure~\ref{fig:example1}.
We also compare this method with the Forward (Euler) method and the Crank-Nicolson method.

\end{example}

\begin{figure}[t]
\centering
\setlength{\unitlength}{\textwidth}
\begin{picture}(1,.3)(0,0)
\put(0,0){\includegraphics[width=.5\unitlength]{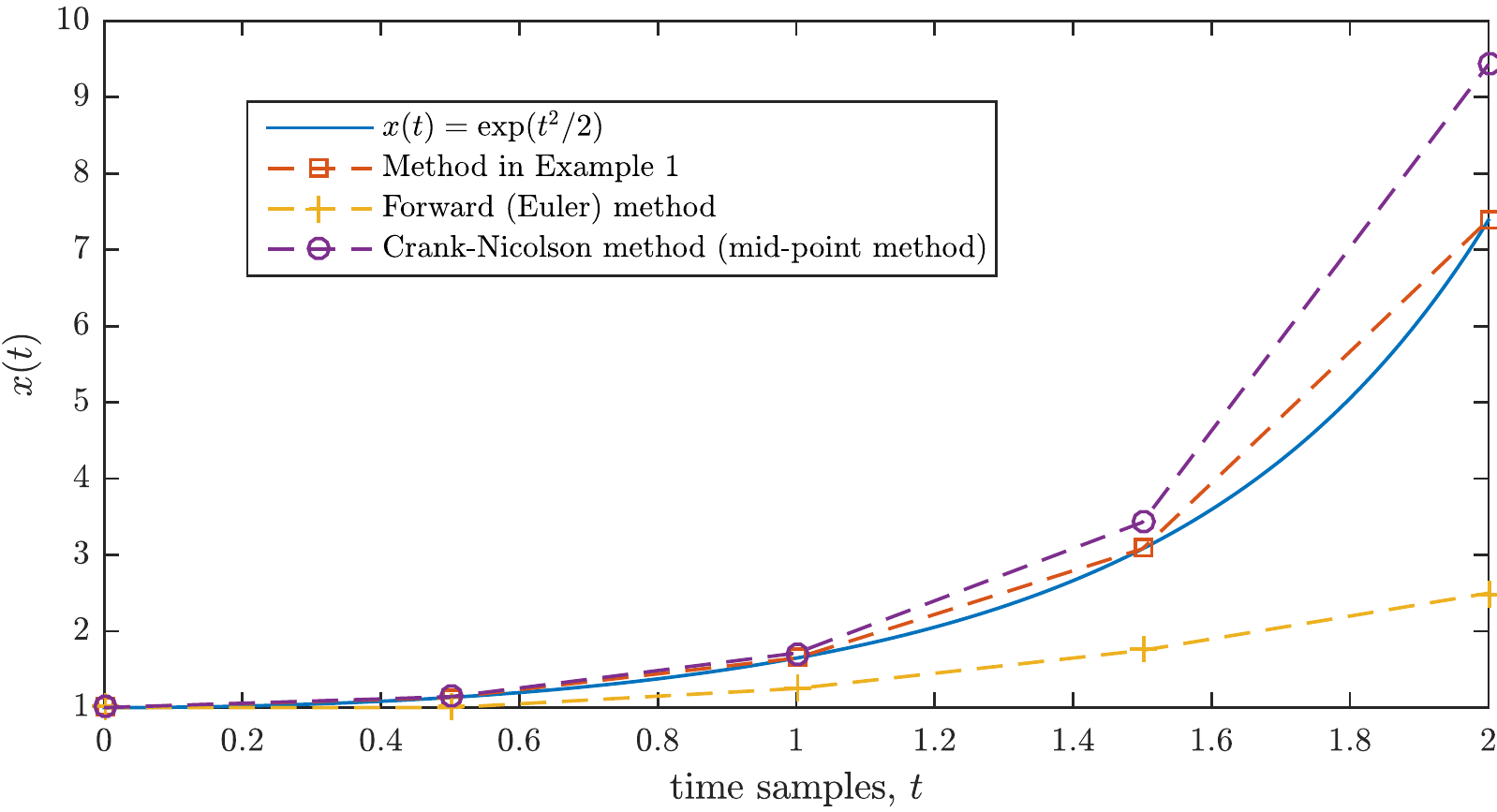}}
\put(0.5,0){\includegraphics[width=.5\unitlength]{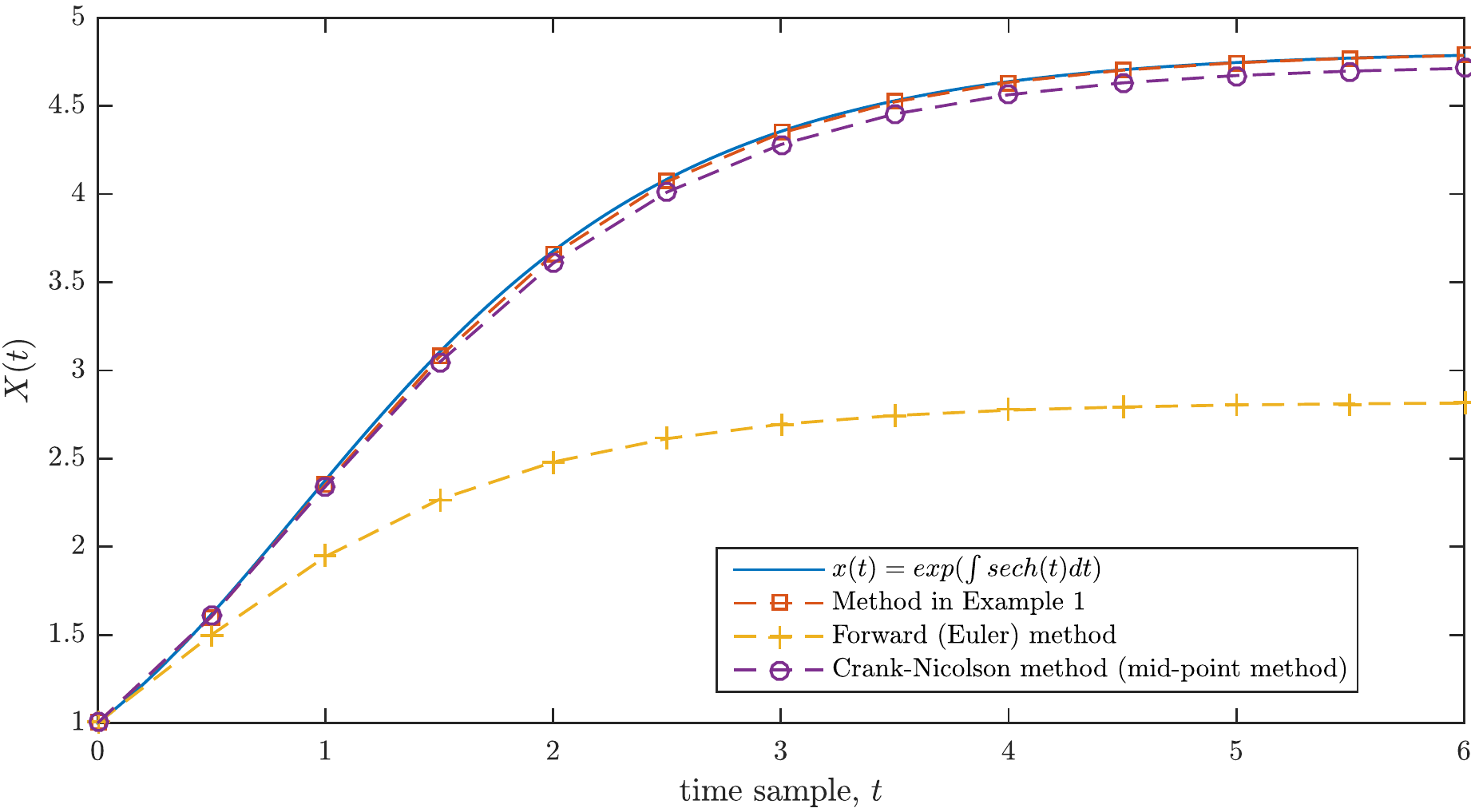}}
\end{picture}
\caption{\label{fig:example1} The comparison of the trapezoid method in Example 1 with Crank-Nicolson method and the forward (Euler) method for two test functions $f(t)=t$ (left figure) and $f(t)=\text{sech}(t)$ (right figure).}
\end{figure}
It would be more convenient if we could write   
the solution of differential equation \eqref{eq:psit}
as follows:
\begin{equation*}
\psi(t;\lambda)=
\exp\left(\int_{-T_0}^t F(\tau;\lambda)\text{d}\tau \right)
\begin{pmatrix}
1\\0
\end{pmatrix}, \text{for } t\geq -T_0 
\end{equation*}
where 
\begin{equation*}
F(\tau;\lambda)=
\begin{pmatrix}
0 & q(\tau)e^{+j2\lambda \tau}\\
-q^*(\tau)e^{-j2\lambda \tau} & 0
\end{pmatrix}.
\end{equation*}
This is not allowed since $F(\tau_1;\lambda)F(\tau_2;\lambda)\neq F(\tau_2;\lambda)F(\tau_1;\lambda)$. Nevertheless, we can still define the step function
$\tilde{F}(t;\lambda)$ which follows the trapezoid rule of integral. Let $h = 2T_0/N$ where $N$ is the
discretization number and $t_n=-T_0+nh$. Define
\begin{equation*}
\tilde{F}(t;\lambda)=
\begin{cases}
F(-T_0;\lambda)&-T_0\leq t<-T_0+h/2\\
F(t_n;\lambda)& t_n-h/2\leq t<t_n+h/2 \text{ and } 1\leq n<N\\
F(T_0;\lambda)& t_N-h/2\leq t \leq T_0
\end{cases}
\end{equation*}
Now consider again \eqref{eq:psit} in the interval $t\in[t_n,t_{n+1})$.
\begin{equation*}
\frac{\partial}{\partial t}\psi(t;\lambda) =
F(t;\lambda)
\psi(t;\lambda)\approx \tilde{F}(t;\lambda) \psi(t;\lambda).
\end{equation*}
Partitioning  $[t_n,t_{n+1})=[t_n,t_n+\frac{h}{2})\cup [t_{n-1}-\frac{h}{2},t_{n+1})$, we have
\begin{equation*}
\psi(t_{n+1};\lambda) \approx\exp(F(t_{n+1};\lambda)\frac{h}{2})
\exp(F(t_{n};\lambda)\frac{h}{2})\psi(t_n;\lambda).
\end{equation*}
The above equation is similar to the equation \eqref{eq:scalerupdate} which is obtained from
 the trapezoid rule of integral for a scalar differential equation. We can define
\begin{equation}
\label{eq:update1}
w_{n+1}=\exp(F(t_{n+1};\lambda)h)w_{n},\;\; w_0=\exp(F(t_0;\lambda)\frac{h}{2})\begin{pmatrix}
1\\0\end{pmatrix}
\end{equation} 
Therefore, $\psi(t_{n})\approx \exp(-F(t_{n};\lambda)\frac{h}{2})w_n$, and 
\begin{equation}
\label{eq:ab_calc}
\begin{pmatrix}
a_N(\lambda)\\b_N(\lambda)
\end{pmatrix}= \exp(-F(t_{N};\lambda)\frac{h}{2})w_N,
\end{equation}
where $a_N(\lambda)$ and $b_N(\lambda)$ are the numerical approximation of $a(\lambda)$
and $b(\lambda)$. As $N\to\infty$, $a_N(\lambda)\to a(\lambda)$ and $b_N(\lambda)\to b(\lambda)$.
Let us now compute $\exp(F(t_n;\lambda)h)$. Let $|q_n|=|q(t_{n})|$ and $e^{j\theta_n}=\frac{q(t_{n})}{|q(t_{n})|}$. We can verify by the eigenvalue decomposition that
\begin{equation*}
F(t_n;\lambda)
=\frac{1}{2}
\begin{pmatrix}
1& j e^{j\theta_n+j2\lambda t_n}\\
j e^{-j\theta_n-j2\lambda t_n}& 1
\end{pmatrix}
\begin{pmatrix}
j|q_n|& 0\\
0& -j|q_n|
\end{pmatrix}
\begin{pmatrix}
1& -j e^{j\theta_n+j2\lambda t_n}\\
-j e^{-j\theta_n-j2\lambda t_n}& 1
\end{pmatrix},
\end{equation*}
and thus,
\begin{align*}
\exp(F(t_n;\lambda)h)
&=\frac{1}{2}
\begin{pmatrix}
1& j e^{j\theta_n+j2\lambda t_n}\\
j e^{-j\theta_n-j2\lambda t_n}& 1
\end{pmatrix}
\begin{pmatrix}
e^{j|q_n|h}& 0\\
0& e^{-j|q_n|h}
\end{pmatrix}
\begin{pmatrix}
1& -j e^{j\theta_n+j2\lambda t_n}\\
-j e^{-j\theta_n-j2\lambda t_n}& 1
\end{pmatrix}\\
&=
\begin{pmatrix}
\cos(|q_n|h)& \sin(|q_n|h)e^{j\theta_n+j2\lambda t_n}\\
-\sin(|q_n|h)e^{-j\theta_n-j2\lambda t_n}&\cos(|q_n|h)
\end{pmatrix}
\end{align*}
Note that $\exp(\pm F(t_n;\lambda)\frac{h}{2})$ is obtained by replacing $h$ with $\pm \frac{h}{2}$
in the above matrix. Consequently,
\begin{equation}
\label{eq:update2}
w_{n}=\begin{pmatrix}
\cos(|q_n|h)& \sin(|q_n|h)e^{j\theta_n+j2\lambda t_n}\\
-\sin(|q_n|h)e^{-j\theta_n-j2\lambda t_n}&\cos(|q_n|h)
\end{pmatrix}
w_{n-1}
\end{equation}

To compute the discrete spectrum, we need to estimate $a'(\lambda)=\frac{d}{d\lambda}a(\lambda)$. Define
\begin{equation*}
G_n=
\begin{pmatrix}
\cos(|q_n|h)&\sin(|q_n|h)e^{j\theta_n+j2\lambda t_n}\\
-\sin(|q_n|h)e^{-j\theta_n-j2\lambda t_n}&\cos(|q_n|h)
\end{pmatrix},
\end{equation*} 
and,
\begin{equation*}
G'_n= +j2t_n\sin(|q_n|h)
\begin{pmatrix}
0&e^{j\theta_n+j2\lambda t_n}\\
e^{-j\theta_n-j2\lambda t_n}&0
\end{pmatrix}.
\end{equation*}
Note that $\exp(\pm F(t_n;\lambda)\frac{h}{2})=G_n^{\pm\frac{1}{2}}$.
Differentiating from \eqref{eq:update1},
\begin{equation}
\label{eq:update_deriv}
w'_{n+1} = G_{n+1}w'_n + G'_{n+1} w_n,\;\;w'_0=(G^{\frac{1}{2}}_0)'\begin{pmatrix}
1\\0
\end{pmatrix}
\end{equation}
where $w'_{n+1}=(\frac{d}{d\lambda}w_{(n+1,1)},\frac{d}{d\lambda}w_{(n+1,2)})^T$.
We can also verify that
\begin{align*}
\psi'(t_n;\lambda)&\approx G_n^{-\frac{1}{2}}\left( w'_n-2jt_n\sin(|q_n|\frac{h}{2})
\begin{pmatrix}
0&e^{j\theta_n+j2\lambda t_n}\\
e^{-j\theta_n-j2\lambda t_n}&0
\end{pmatrix} G_n^{-\frac{1}{2}}w_n\right)\\
&= G_n^{-\frac{1}{2}}w'_n -2jt_n\sin(|q_n|\frac{h}{2})
\begin{pmatrix}
0&e^{j\theta_n+j2\lambda t_n}\\
e^{-j\theta_n-j2\lambda t_n}&0
\end{pmatrix} w_n,
\end{align*}
where 
\begin{equation*}
G_n^{-\frac{1}{2}}=
\begin{pmatrix}
\cos(|q_n|\frac{h}{2})&-\sin(|q_n|\frac{h}{2})e^{j\theta_n+j2\lambda t_n}\\
\sin(|q_n|\frac{h}{2})e^{-j\theta_n-j2\lambda t_n}&\cos(|q_n|\frac{h}{2})
\end{pmatrix}.
\end{equation*}
Finally, we have
\begin{equation}
\label{eq:da_calc}
\begin{pmatrix}
a'(\lambda)\\b'(\lambda)
\end{pmatrix}
\approx
\begin{pmatrix}
a_N'(\lambda)\\b_N'(\lambda)
\end{pmatrix}
=
G_N^{-\frac{1}{2}}w'_N -2jt_N\sin(|q_N|\frac{h}{2})
\begin{pmatrix}
0&e^{j\theta_N+j2\lambda t_N}\\
e^{-j\theta_N-j2\lambda t_N}&0
\end{pmatrix} w_N.
\end{equation}
We can therefore numerically compute $a(\lambda)$, $b(\lambda)$ and $a'(\lambda)$ from \eqref{eq:update1}, \eqref{eq:ab_calc}, \eqref{eq:update_deriv} and \eqref{eq:da_calc}.

\begin{remark}
The final recursion \eqref{eq:update2} looks like the layer-peeling algorithm in \cite{yousefi2013thesis}. However, there are some small differences. The layer-peeling algorithm
is a kind of forward (Euler) algorithms and obtained from approximating $q(t)$ as a piece-wise constant function. Our algorithm is a kind of mid-point algorithms (Trapezoid rule) which usually provides better approximations (see Figure~\ref{fig:example1}). Moreover, we approximate the
matrix $F(t)$ as a piece-wise constant function. The other practical advantage is the equation
\eqref{eq:update2} has a much simpler form than the layer-peeling algorithm.
\end{remark}

\begin{remark}
The first order approximation of $G_n$ gives the forward (Euler) discretizations method~\cite{yousefi2013thesis}. When $|q_n|h\ll 1$, $\cos(|q_n|h)\approx 1$ and $\sin(|q_n|h)\approx|q_n|h$ and thus,
\begin{equation*}
G_n\approx
\begin{pmatrix}
1&h q(t_n)e^{j2\lambda t_n}\\
-hq^*(t_n)e^{-j2\lambda t_n}&1
\end{pmatrix}.
\end{equation*}
The second order approximation of $G_n$ gives the promising Crank-Nicolson method~\cite{yousefi2013thesis}. 
Using the relations $\cos(2\theta)=(1-\tan^2(\theta))/(1+\tan^2(\theta))$ and  
$\sin(2\theta)=(2\tan(\theta))/(1+\tan^2(\theta))$ and approximation $\tan(|q_n|\frac{h}{2})\approx|q_n|\frac{h}{2}$, we have
\begin{equation*}
G_n\approx
\frac{1}{1+\frac{h^2}{4}|q_n|^2}
\begin{pmatrix}
1-\frac{h^2}{4}|q_n|^2&h q(t_n)e^{j2\lambda t_n}\\
-hq^*(t_n)e^{-j2\lambda t_n}&1-\frac{h^2}{4}|q_n|^2
\end{pmatrix},
\end{equation*}
which is the normalized matrix of the Crank-Nicolson method for solving \eqref{eq:psit}, i.e.
\begin{equation*}
\frac{\psi(t_{n+1};\lambda)-\psi(t_{n};\lambda)}{h}\approx
\frac{1}{2}\left( F(t_n;\lambda)\psi(t_{n};\lambda) +F(t_{n+1};\lambda)\psi(t_{n+1};\lambda)\right).
\end{equation*} 
\end{remark}

\subsection{Forward-Backward Method}

Our method and all the one-step discretization NFT methods explained in \cite{yousefi2013thesis}, e.g.
Ablowitz-Ladik algorithms, Crank-Nicolson method, Runge-Kutta method, Layer-peeling method,
are robust in estimation of the continuous spectrum and
can be used to find the eigenvalues ($a(\lambda)=0$) with a small numerical error. However, 
they are very naive in estimation of the discrete spectral amplitudes, $Q_d(\lambda_i)$. The
estimation error may exponentially grow with a slight change of step-size $h$.

Let us intuitively explain how a small numerical error in $a(\lambda)$ may lead to an exponentially
large error in estimating $b(\lambda)$. Consider the temporal equation \eqref{eq:psit} at an eigenvalue of $q(t)$, 
namely $\lambda=\omega + j\eta$ with $\eta>0$. We know that $a(\lambda)=0$ and
$b(\lambda)$ is a finite complex number. Assume that a numerical method estimates
$\bar{\psi}(T_0-1)=\psi(T_0-1)+(\delta_1,\delta_2)^T$ where $(\delta_1,\delta_2)^T$
is its estimation error. 
Now look at \eqref{eq:psit} for $T_0-1<t<T_0$. The changes in $\frac{\partial}{\partial t}\psi_1$
is proportional to $|q(t)|\exp(-2\eta t)\delta_2$ which is very small when $T_0\gg 1$
and thus, $a(\omega + j\eta)$ is finally approximated by $\bar{\psi_1}(T_0)\approx\delta_1$.
However,
the changes in $\frac{\partial}{\partial t}\psi_2$ is
proportional to $|q(t)|\exp(+2\eta t)\delta_1$ which 
can be very large for a large $\eta$. It means that the gap $\bar{\psi}_2(T_0)-b(\lambda)$
can become exponentially large for a small error $\delta_1$ in $a(\lambda)$.
This issue is shown in Figure~\ref{fig:example2}-c. We observe that
$\bar{\psi}_2$ diverges exponentially from the exact solution $\psi_2(t)$ when $h$ is not
small enough.

Let us explain how to significantly stabilize an one-step NFT method by the
forward-backward algorithm. We explain this algorithm based on the trapezoid discretizations
method but the algorithm can be applied to any one-step method. Consider the equation \eqref{eq:update2}. We can write,
\begin{align*}
\begin{pmatrix}
a_N(\lambda)\\
b_N(\lambda)
\end{pmatrix} &= G_N^{-\frac{1}{2}} G_N G_{N-1}\cdots G_{2}G_{1} G_0^{+\frac{1}{2}}
\begin{pmatrix}
1\\
0
\end{pmatrix}\\
&=R_NL_N
\begin{pmatrix}
1\\
0
\end{pmatrix},
\end{align*}
where $R_N=G_N^{-\frac{1}{2}} G_N G_{N-1}\cdots G_{m+1}$ and $L_N=G_m G_{m-1}\cdots G_1 G_0^{+\frac{1}{2}}$ and the integer $m=cN$, $0<c<1$. 
First, let $N\to\infty$. Then, $(a_N(\lambda),b_N(\lambda))\to(a(\lambda),b(\lambda))$ and
$L_N\to L$ where $L$ is the scattering matrix of $q(t)\mathbf{1}\{-T_0<t<T_0(2c-1)\}$
and $R_N\to R$ where $R$ is the scattering matrix of $q(t)\mathbf{1}\{T_0(2c-1)\leq t<T_0\}$. Let $R_{kl}$ and $L_{kl}$ denote the entry $(k,l)$ of matrices $R$ and $L$. Then,
\begin{align*}
a(\lambda)&=R_{11}L_{11} + R_{12}L_{21}\\
a'(\lambda)&=L_{11}\frac{d}{d\lambda}R_{11}+
R_{11}\frac{d}{d\lambda}L_{11} +
L_{21}\frac{d}{d\lambda}R_{12} +
R_{12}\frac{d}{d\lambda}L_{21}\\
b(\lambda)&=R_{21}L_{11} + R_{22}L_{21}
\end{align*}
The determinant of matrices $G_n$ is one. 
It implies that $R$ has determinant one, 
i.e. $R_{11}R_{22}-R_{12}R_{21}=1$. 
Therefore,
\begin{equation*}
R^{-1}=G_{m+1}^{-1}G_{m+2}^{-1}\cdots G_{N-1}^{-1}
G_{N}^{-1}G_{N}^{\frac{1}{2}}=
\begin{pmatrix}
R_{22}&-R_{12}\\
-R_{21}&R_{11}
\end{pmatrix}.
\end{equation*} 
Let $G_{n,kl}$ denote the entry $(k,l)$ of matrix $G_n$. One can verify that 
if $\lambda\in\mathbb{R}$, then $G_{n,22}=G^*_{n,11}$ 
and $G_{n,21}=-G^*_{n,12}$. 
It implies that $R_{22}=R^*_{11}$ and $R_{21}=-R^*_{12}$ for $\lambda\in\mathbb{R}$.
We have
\begin{align*}
\begin{pmatrix}
L_{11}\\
L_{21}
\end{pmatrix} 
 = L\begin{pmatrix}
1\\
0
\end{pmatrix},
\begin{pmatrix}
-R_{12}\\
R_{11}
\end{pmatrix} 
 = R^{-1}\begin{pmatrix}
0\\
1
\end{pmatrix}.
\end{align*}
We already showed $a(\lambda)$ and $a'(\lambda)$ as a function of $L_{11}$, $L_{21}$, $R_{11}$ and $R_{12}$.
We can also represent $b(\lambda)$ in terms of these parameters.
\begin{align*}
b(\lambda)&=R_{21}L_{11} + R_{22}L_{21}\\
&=R_{21}L_{11} + \frac{1+R_{12}R_{21}}{R_{11}}L_{21}\\
&=\frac{R_{21}(R_{11}L_{11} + R_{12}L_{21})+L_{21}}{R_{11}}\\
&=a(\lambda)\frac{R_{21}}{R_{11}} + \frac{L_{21}}{R_{11}}.
\end{align*}
For $\lambda\in\mathbb{R}$,
\begin{equation}
\label{eq:b_real}
b(\lambda) = -a(\lambda)\frac{R^*_{12}}{R_{11}} + \frac{L_{21}}{R_{11}}.
\end{equation}
For eigenvalue $\lambda_i$ of $q(t)$ (i.e. $a(\lambda_i)=0$),
\begin{equation}
\label{eq:b_eigen}
b(\lambda_i) = \frac{L_{21}}{R_{11}}.
\end{equation}
\begin{remark}
In Equation \eqref{eq:b_eigen},
we remove the contribution of $a(\lambda_i)\frac{R_{21}}{R_{11}}$ from $b(\lambda_i)$ as $a(\lambda_i)=0$. This part causes a large numerical error in one-step methods. One can verify that
$R_{21}$ is very large when $\eta=\frac{1}{2}(\lambda-\lambda^*)>0$. Then small numerical error $a_N(\lambda_i)=\delta_1$ leads to a large 
numerical error $\delta_1\frac{R_{21}}{R_{11}}$ in estimation $b_N(\lambda_i)$. 
\end{remark}
Now consider a finite $N$. We estimate $(L_{11},L_{21})$
and $(-R_{12},R_{11})$ from two separate recursions: The \emph{forward
recursion} computes $(L_{11},L_{21})$ and its derivative. For $1\leq n\leq m$,
\begin{align*}
w_{n}&=G_{n}w_{n-1},\; w_0=G_0^{\frac{1}{2}}\begin{pmatrix}
1\\0
\end{pmatrix}\\
w'_{n} &= G_{n}w'_{n-1} + G'_{n} w_{n-1},\;\;w'_0=(G^{\frac{1}{2}}_0)'\begin{pmatrix}
1\\0
\end{pmatrix}
\end{align*} 
Above the equations are the same as \eqref{eq:update2} and \eqref{eq:update_deriv} in which $G_n$, $G'_n$,
$G_0^{\frac{1}{2}}$ and $(G_0^{\frac{1}{2}})'$ are defined. Thus, 
\begin{equation}
\label{eq:L1L2}
\begin{pmatrix}
L_{11}\\
L_{21}
\end{pmatrix}\approx w_m=
\begin{pmatrix}
w_{m,1}\\
w_{m,2}
\end{pmatrix},
\begin{pmatrix}
\frac{d}{d\lambda}L_{11}\\
\frac{d}{d\lambda}L_{21}
\end{pmatrix}\approx
w'_m=
\begin{pmatrix}
w'_{m,1}\\
w'_{m,2}
\end{pmatrix}
\end{equation}
The \emph{backward recursion} computes $(-R_{12},R_{11})$
and its derivative. For $m<n\leq N$,
\begin{align*}
v_{n-1}&=G^{-1}_n v_{n},\; v_N=G_N^{\frac{1}{2}}
\begin{pmatrix}
0\\1
\end{pmatrix}\\
v'_{n-1}&=G^{-1}_n v'_{n}+(G^{-1}_n)'v_{n},\; v'_N=
(G_N^{\frac{1}{2}})'
\begin{pmatrix}
0\\1
\end{pmatrix}
\end{align*}
$G^{-1}_n$ is obtained by replacing $h$ with $-h$ in
$G_n$ as follows:
\begin{equation*}
G^{-1}_n=
\begin{pmatrix}
\cos(|q_n|h)&-\sin(|q_n|h)e^{j\theta_n+j2\lambda t_n}\\
+\sin(|q_n|h)e^{-j\theta_n-j2\lambda t_n}&\cos(|q_n|h)
\end{pmatrix},
\end{equation*} 
and,
\begin{equation*}
(G^{-1}_n)'= -j2t_n\sin(|q_n|h)
\begin{pmatrix}
0&e^{j\theta_n+j2\lambda t_n}\\
e^{-j\theta_n-j2\lambda t_n}&0
\end{pmatrix}.
\end{equation*} 
As a result,
\begin{equation}
\label{eq:R1R2}
\begin{pmatrix}
-R_{12}\\
R_{11}
\end{pmatrix}\approx v_m=
\begin{pmatrix}
v_{m,1}\\
v_{m,2}
\end{pmatrix},
\begin{pmatrix}
-\frac{d}{d\lambda}R_{12}\\
\frac{d}{d\lambda}R_{11}
\end{pmatrix}\approx
v'_m=
\begin{pmatrix}
v'_{m,1}\\
v'_{m,2}
\end{pmatrix}
\end{equation}
From \eqref{eq:L1L2} and \eqref{eq:R1R2}, we have
\begin{align*}
a_N(\lambda)&=w_{m,1}v_{m,2}-v_{m,1}w_{m,2}\\
a'_N(\lambda)&=w'_{m,1}v_{m,2}+w_{m,1}v'_{m,2}-v'_{m,1}w_{m,2}-v_{m,1}w'_{m,2}
\end{align*}
and,
\begin{align*}
b_N(\lambda)&=a_N(\lambda)\frac{v^*_{m,1}}{v_{m,2}}
+ \frac{w_{m,2}}{v_{m,2}},\;\;\text{for } \lambda\in\mathbb{R},\\
b_N(\lambda_i) &= \frac{w_{m,2}}{v_{m,2}},\;\;\text{for } \lambda_i\in\mathbb{C^+} \text{ the eigenvalues of }q(t)
\end{align*}
The main question is how to choose $m=cN$. 
Consider the update matrix $G_n$ for $\lambda=w+j\eta$: $|G_{n,12}|=|q(t_n)|\exp(-2\eta t_n)$ and $|G_{n,21}|=|q(t_n)|\exp(+2\eta t_n)$. The scattering parameter $L_{21}$ is somehow the integral of $G_{n,21}$
and thus, the numerical estimation error in $L_{21}$ becomes large when $G_{n,21}$ is large.
A similar behaviour is between $R_{12}$ and $G_{n,12}$. Therefore, one can choose
\begin{align*}
c=\text{arg}\min_{t=T_0(2c-1)}|q(t)|\exp(2\eta |t|).
\end{align*}
In the following examples, we fixed $c=1/2$ (i.e. $m=N/2$).
\begin{example}[2-solitonic pulse]
Consider a 2-solitonic pulse with eigenvalues $\lambda_1=0.5j$ and $\lambda_2=1j$ and with
$Q_d(\lambda_1)=3$ and $Q_d(\lambda_2)=-6$. This pulse $q(t)$ is symmetric and has only imaginary
component. We truncate the pulse outside the interval $t\in[-5,5]$ in which contains more than $99.9\%$ of pulse energy. The pulse is plotted in Figure~\ref{fig:example2}a.

We uniformly sample the truncated pulse by $N$ samples and apply the following numerical NFT algorithm: Forward-backward algorithm (FB), Trapezoid discretizations (TD),
Crank-Nicolson algorithm (CN) and Ablowitz-Ladik discretizations algorithm (AL). 
Table~\ref{table:2soliton} shows the discrete spectral amplitudes computed by each algorithm in terms of $N$. The forward-backward algorithm is much more precise and stable than the others. 
In particular, it approximates the discrete spectrum with an acceptable precision when $N$ is 
as small as $32$. We plot the numerical estimation of $\psi(t_n)$ from temporal equation
\eqref{eq:psit} using the FB algorithm and AL discretizations algorithm
in Figure~\ref{fig:example2}b-d when $N=64$. We plot also the exact solution by letting $N=10^6$. 
One can see how the estimations of the FB algorithm 
closely follow the exact solution for both eigenvalues.

\begin{table}
\caption{\label{table:2soliton} The discrete spectral amplitudes of a 2-solitonic pulse from different algorithms in terms of $N$.} 
\centering 
\begin{tabular}{c| c c c c| c c c c |c c c c}
\hline\hline 
Spectrum & \multicolumn{4}{c|}{N=32}& \multicolumn{4}{c|}{ $N=64$ }& \multicolumn{4}{c}{ $N=1024$ }\\
 & AL & CN & TD & FB & AL & CN & TD & FB & AL & CN & TD & FB \\
\hline
 $a(\lambda=0.5j)$   & 0.12 &  3.0e-2 &  -1.3e-2   &  -1.3e-2  &  4.0e-2  &  8.3e-3   &  -3.2e-3   &  -3.2e-3  & -9.9e-5 &  -2.4e-4  &  -2.8e-4  &  -2.8e-4  \\
$Q_d(\lambda=0.5j)$  & \textbf{-4.34}  &  \textbf{1.15}   & \textbf{ 3.47}   & \textbf{ 2.75}  &  \textbf{0.825}  &  \textbf{ 2.48}  &   \textbf{3.13}  &   \textbf{2.94 } & \textbf{ 3.03}  &  \textbf{3.02}  &  \textbf{3.04}  &  \textbf{3.01 } \\
 $a(\lambda=1j)$    &  0.11   &  1.3e-2   &  -2.9e-2   &  -2.9e-2  &  3.4e-2  &  4.1e-3   &  -6.9e-3  &  -6.9e-3  &  1.5e-4  &  1.6e-5  &  -2.8e-5  &  -2.8e-5 \\
$Q_d(\lambda=1j)$  &  \textbf{819.3}   &  \textbf{78.57}   &  \textbf{-197.1}   &  \textbf{-7.589}  &  \textbf{189.4}  &  \textbf{15.46}   &  \textbf{-42.35}  &  \textbf{-6.292}   &  \textbf{-5.318}  &  \textbf{-5.913}  &  \textbf{-6.130}  &  \textbf{-5.991} \\
\hline
\end{tabular}
\end{table}

\end{example}

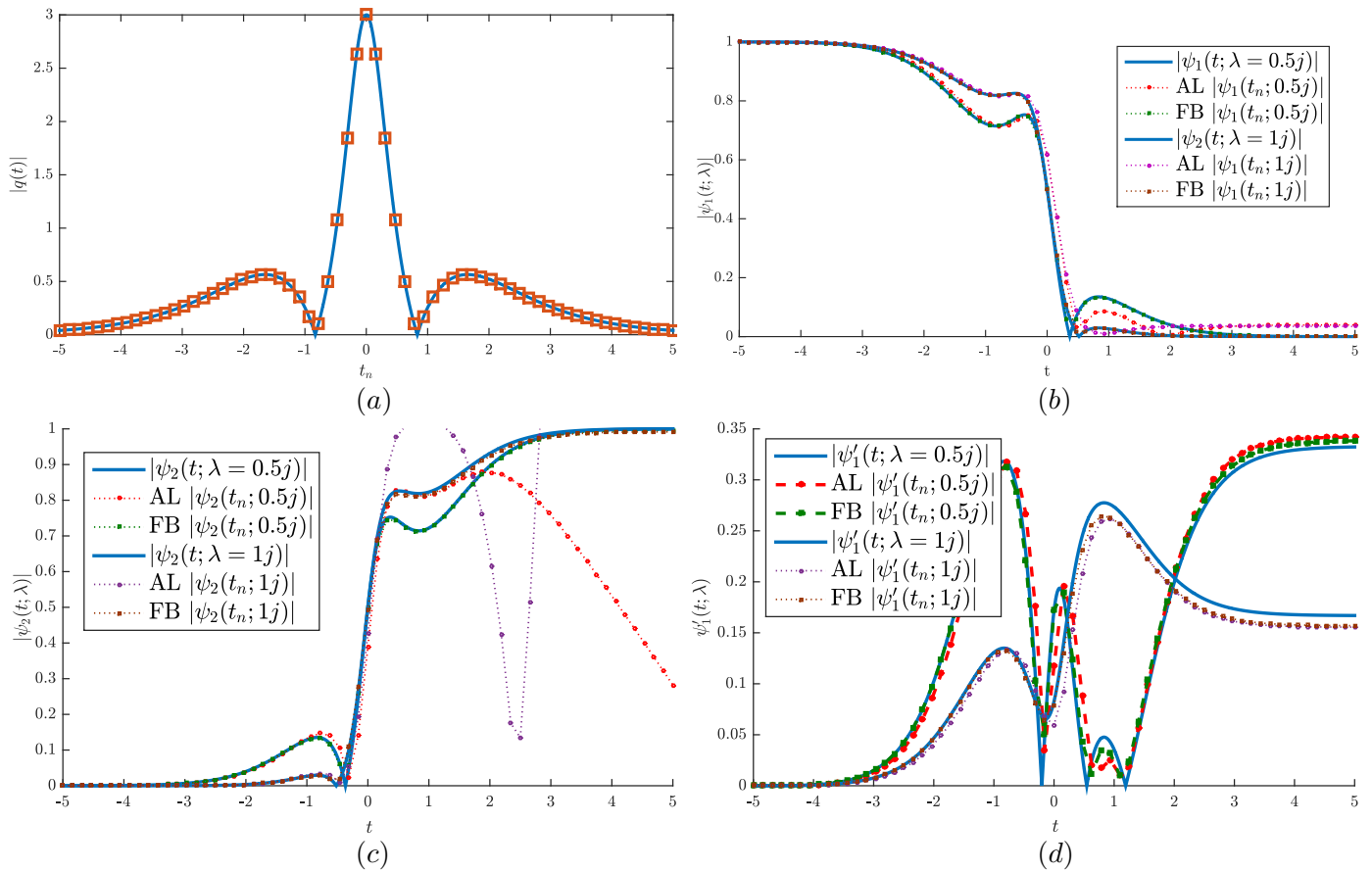
\begin{figure}[t]
\centering
\setlength{\unitlength}{\textwidth}
\begin{picture}(1,.7)(0,0)
\put(-0.01,0.36){\includegraphics[width=.5\unitlength]{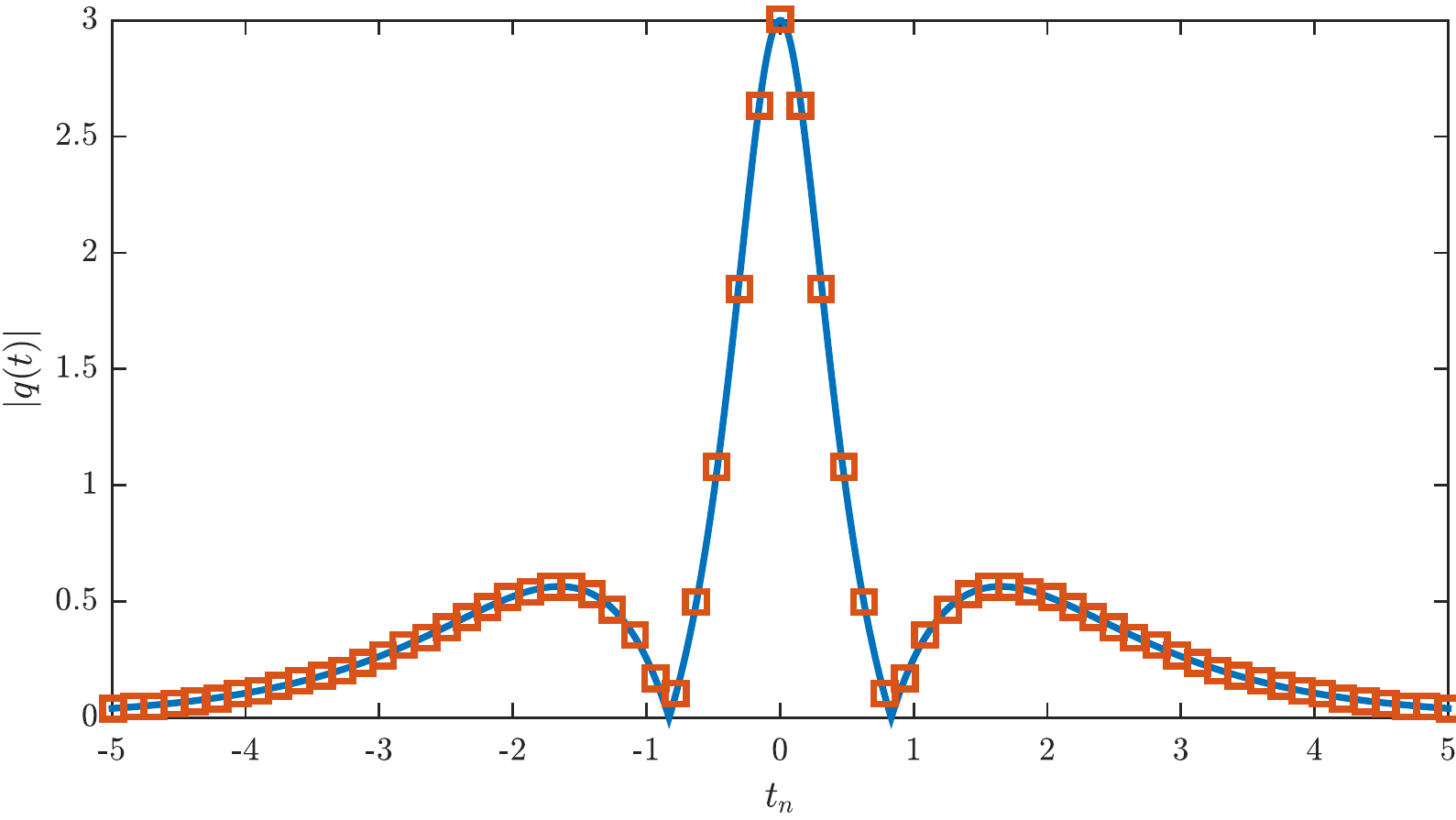}}
\put(0.5,0.36){\includegraphics[width=.5\unitlength]{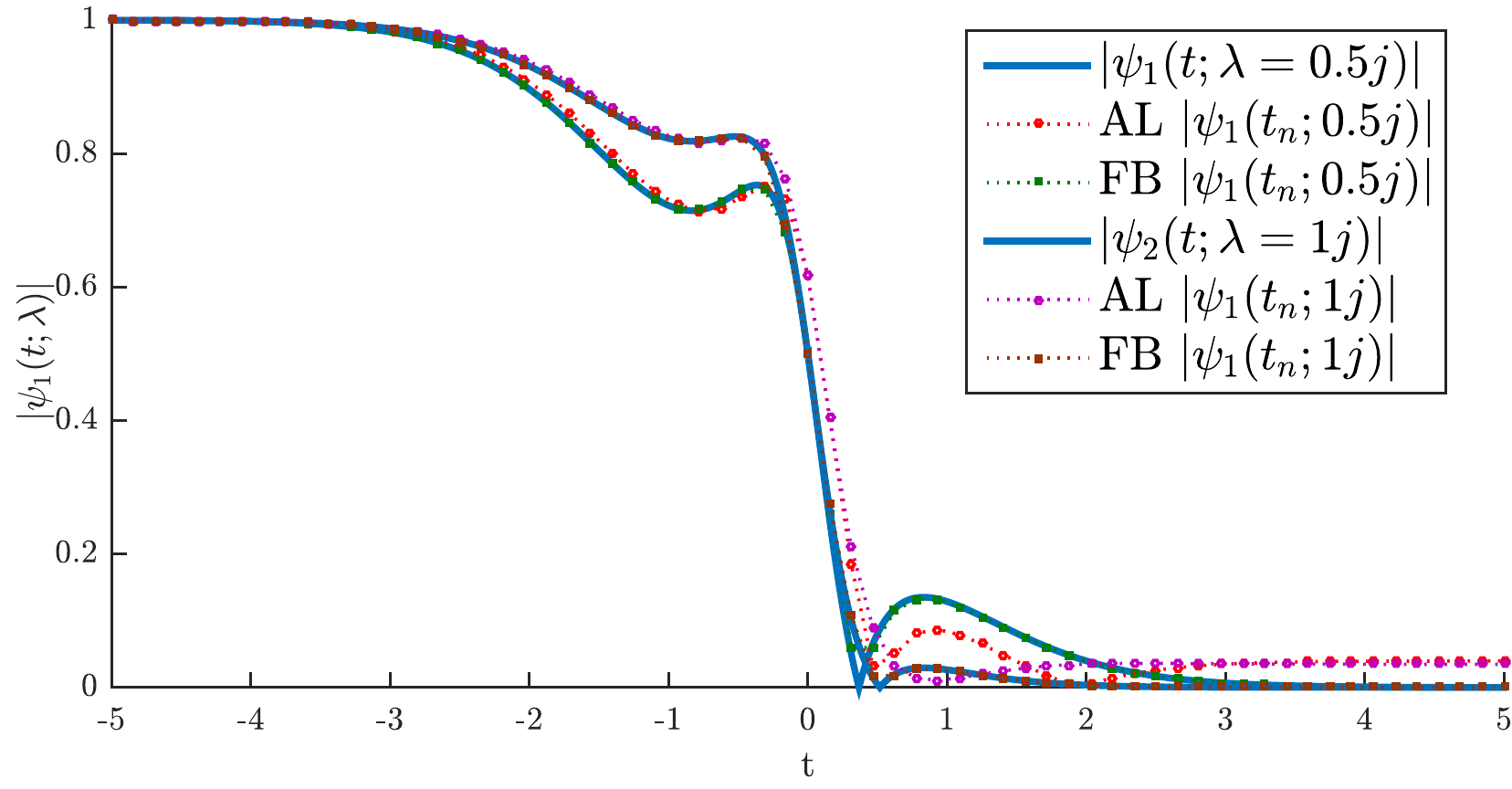}}
\put(-0.01,0.02){\includegraphics[width=.5\unitlength]{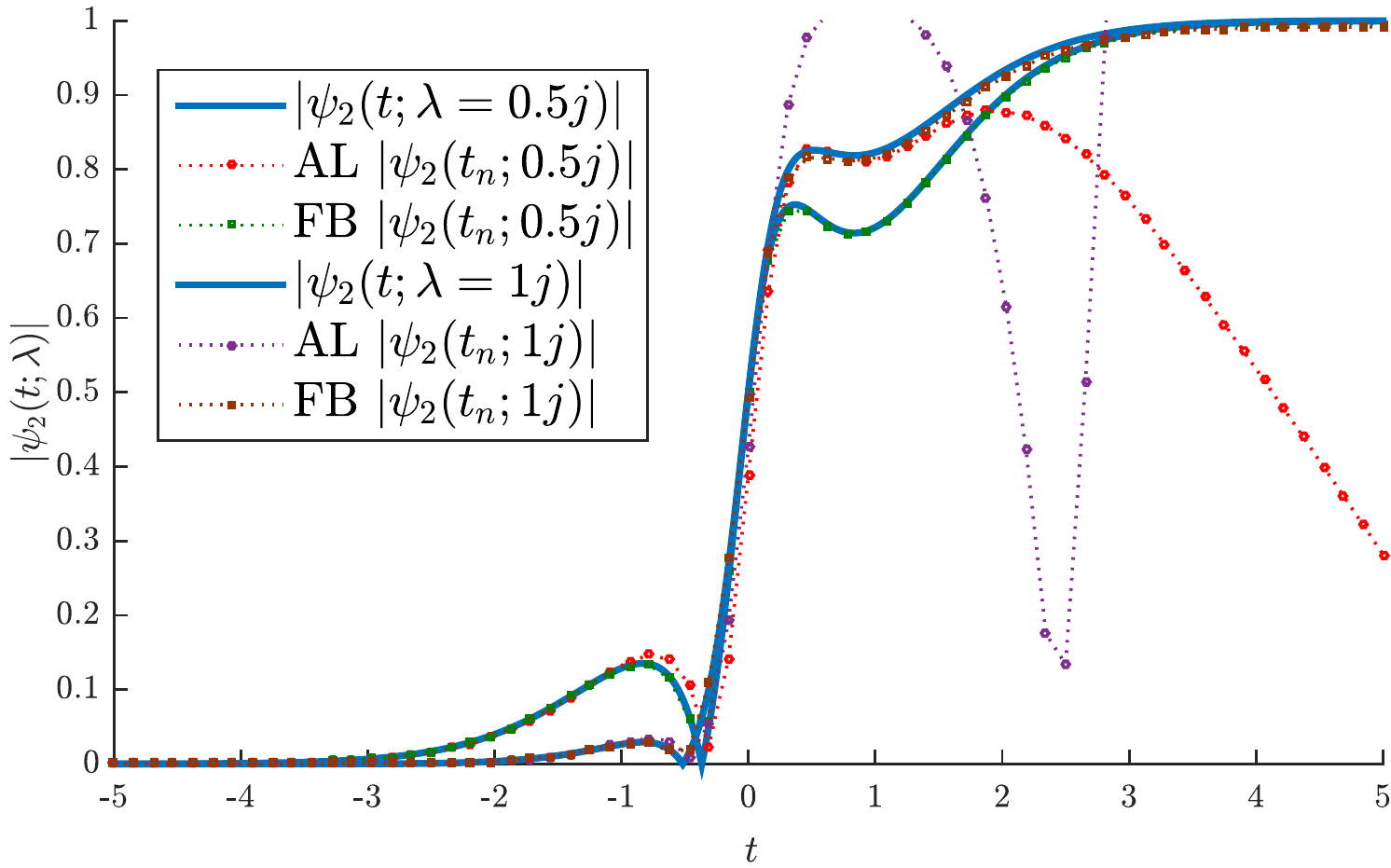}}
\put(0.5,0.02){\includegraphics[width=.5\unitlength]{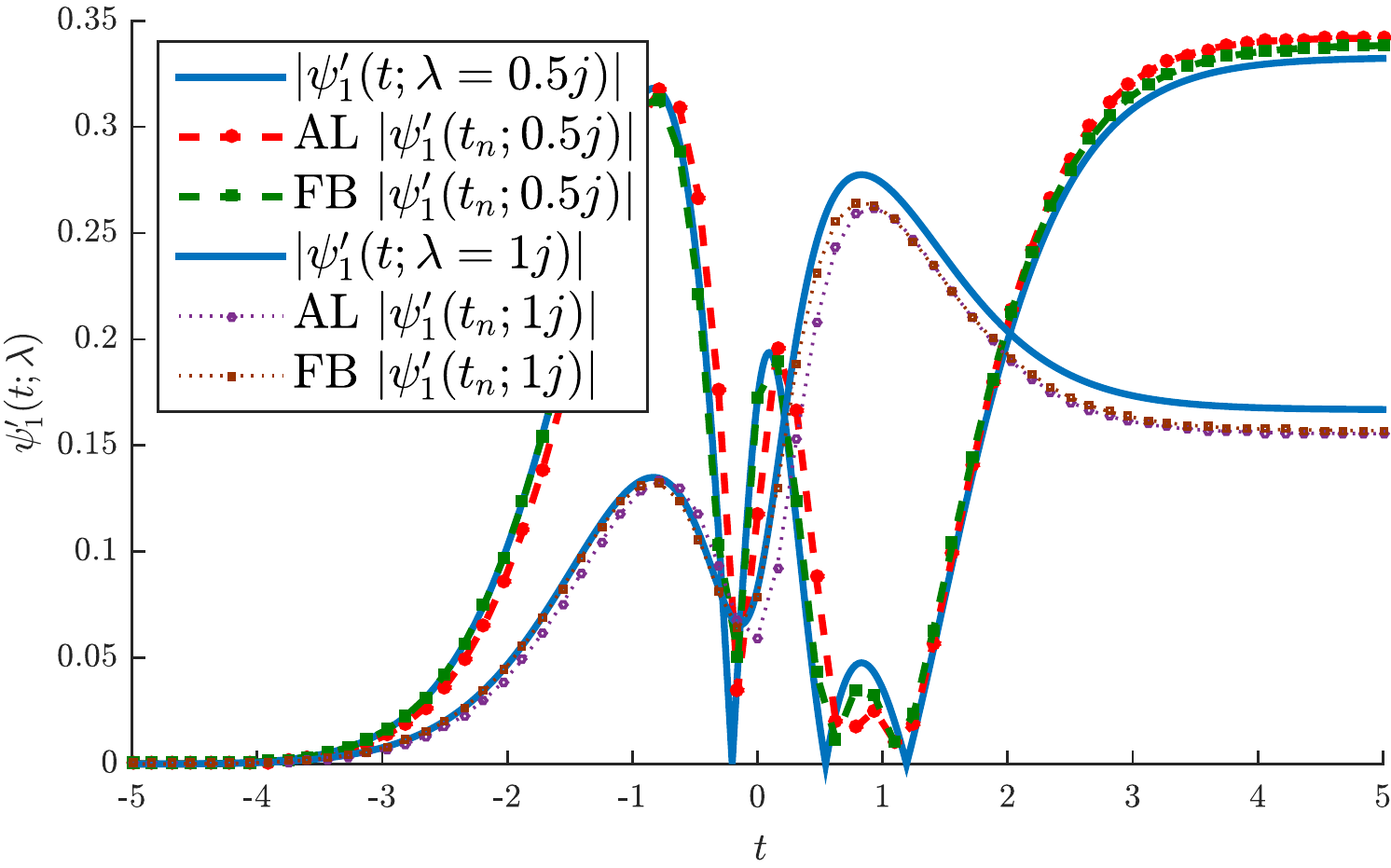}}
\put(0.25,0.34){$(a)$}
\put(0.76,0.34){$(b)$}
\put(0.25,0.0){$(c)$}
\put(0.76,0.0){$(d)$}
\end{picture}
\caption{\label{fig:example2} $(a)$ The absolute value of 2-solitonic pulse $q(t)$ of Example 2 with $N=64$ samples. 
The solution of temporal equation \eqref{eq:psit}, $\psi(t;\lambda)=(\psi_1(t;\lambda),\psi_2(t;\lambda))$, is numerically computed using 
Ablowitz-Ladick (AL) algorithm and forward-backward (FB) algorithm from $N=64$ samples of $q(t)$. The solid blue curves are the exact solutions:$(b)$ $\psi_1(t;\lambda)$, $(c)$ $\psi_2(t;\lambda)$,
$(d)$ $\frac{d}{d\lambda}\psi_1(t;\lambda)$.}
\end{figure}


%

%
%
%
%
%
%

\ifCLASSOPTIONcaptionsoff
  \newpage
\fi

\bibliographystyle{IEEEtran}
\bibliography{references_nft}

\end{document}